\documentclass[reqno]{amsart}
\usepackage{amsmath,amsthm,amssymb}
\usepackage{latexsym}
\usepackage{eucal}
\usepackage{fullpage}
\newtheorem{theorem}{Theorem}[section]
\newtheorem{lemma}{Lemma}[section]
\newtheorem{corollary}{Corollary}[section]

\newtheorem{proposition}{Proposition}[section]
\newcommand{\sgn}{\mbox{sgn}}
\newcommand{\R}{\mathbb{R}}

\theoremstyle{definition}
\newtheorem{definition}{Definition}[section]

\def\cal{\mathcal}

\def\ge{\geqslant}\def\le{\leqslant}
\def\e{\varepsilon}\def\~{\widetilde}

\begin{document}
\title[2D Euler equation on the strip: stability of a rectangular patch
 ]{2D Euler equation on the strip: stability of a rectangular patch}
\author{J. Beichman, S. Denisov  }

\address{
\begin{flushleft}
\vspace{1cm} Jennifer Beichman: jbeichman@smith.edu\\ \vspace{0.1cm}
Smith College\\
Department of Mathematics and Statistics\\
44 College Lane\\
Northampton, MA 01063, USA\\
\vspace{1cm} Sergey Denisov: denissov@wisc.edu\\\vspace{0.1cm}
University of Wisconsin--Madison\\  Mathematics Department\\
480 Lincoln Dr., Madison, WI, 53706,
USA\vspace{0.1cm}\\and\\\vspace{0.1cm}
Keldysh Institute for Applied Mathematics, Russian Academy of Sciences\\
Miusskaya pl. 4, 125047 Moscow, RUSSIA\\
\end{flushleft}
}\maketitle

\renewcommand{\thefootnote}{\fnsymbol{footnote}}
\renewcommand{\thefootnote}{\arabic{footnote}}

\begin{abstract}
We consider the 2D Euler equation of incompressible fluids on a
strip $\mathbb{R}\times \mathbb{T}$ and prove the stability of the
rectangular stationary state $\chi_{|x|<L}$ for large enough $L$.
\end{abstract}\vspace{1cm}

\section{Introduction}
In this paper we will consider the stability of a certain class of steady solutions to the Euler equation in a
 two dimensional cylindrical domain. The study of such stability questions is well developed in the planar case.
 In the plane, the primary focus has been in the stability of circular patches, starting with \cite{wp}, which studied
 the evolution of a circular patch in a bounded domain and proved stability using a spectral argument. In a similar vein,
 \cite{isg} used conserved quantities to derive a bound on the diameter growth. The use of conserved quantities was also
  key in \cite{sv}. There has also been work studying the stability and instability of other steady solutions in the plane, such
  as the Kirchhoff ellipse cf. \cite{sg}.

In other domains, these types of questions are less well understood.
In the strip $\mathbb{R}\times [0,a]$, the work of Caprino and
Marchioro \cite{cm} shows the stability of monotonically increasing
steady vorticity distributions with restrictive conditions on the
associated velocity. More recently, Bedrossian and Masmoudi
\cite{bm} showed nonlinear stability of Couette flow in the cylinder
$S=\mathbb{R}\times \mathbb{T}$.

This paper considers steady patch solutions of the form $\chi_{E_0}(z)$ where $E_0 = [-L,L]\times \mathbb{T}$ to the problem
\begin{equation}\label{ecyl}
\partial_t\theta+\nabla \theta\cdot u=0, \quad \theta|_{t=0}=\theta_0=\chi_E
\end{equation}
for a compact set $E\subset S$. The velocity $u(z,t)$ is related to
the vorticity $\theta$ via a cylindrical Biot-Savart law. Let the
stream function $\Psi$ be the function that solves the elliptic
problem
\begin{equation}\label{strcyl}
(2\pi)^{-1} \Delta \Psi = \theta,\quad \lim_{x\rightarrow
+\infty}\partial_1 \Psi(x,y,t) = - \lim_{x\rightarrow
-\infty}\partial_1 \Psi(x,y,t), \quad |\Psi(x,y,t)|\le C(|x|+1).
\end{equation} The cylindrical Biot-Savart law is then
\begin{equation}\label{bscyl}
u =\nabla^\perp\Psi= k*\theta,  \quad \Gamma(x,y) = \frac{1}{2}\log
(\cosh(x)-\cos(y)),\quad k(x,y) = \nabla^\perp\Gamma(x,y) =
\frac{(-\sin(y),\sinh(x))}{2(\cosh(x)-\cos(y))}.
\end{equation}
%need a lemma asserting the uniqueness of the Biot-Savart Law as we defined it.
The velocity $u$ defined by the cylindrical Biot-Savart law exists
and is unique by the following Lemma.
\begin{lemma}
 Given compactly supported $\theta(z,t)\in L^\infty(S)$, all solutions to equation \eqref{strcyl}
 are given by
\[\Psi(x,y) = \Gamma*\theta +C\] for some constant $C$.
\end{lemma}
\begin{proof}
To see uniqueness, let $\Psi_1$ and $\Psi_2$ be two solutions to the elliptic problem \eqref{strcyl}. Then $\widetilde{\Psi}=\Psi_1-\Psi_2$ solves
\[\Delta \widetilde{\Psi} = 0,\quad
\lim_{x\rightarrow +\infty}\partial_1 \widetilde{\Psi}(x,y,t) = -
\lim_{x\rightarrow -\infty}\partial_1 \widetilde{\Psi}(x,y,t).
\]
The upper bound on $\Psi_1$ and $\Psi_2$ guarantees that the
$\widetilde{\Psi}$ grows at most linearly. If we extend
$\widetilde{\Psi}$ to all of $\R^2$ by periodicity in $y$,
Liouville's Theorem for harmonic functions gives that
$\widetilde{\Psi}=C_1x+C_2y+C_3$. Since $\widetilde{\Psi}$ is
periodic in $y$, $C_2 = 0$.  The conditions on
$\partial_1\widetilde{\Psi}$ mean that $C_1=0$. Thus
$\widetilde{\Psi}=C$.

It remains to show that $\Gamma*\theta$ is a solution to
\eqref{strcyl}. One can directly check that $\Delta \Gamma=0,\,
z\neq 0$. Moreover,
\[
(2\pi)^{-1}\Gamma(z)=\frac{1}{\pi}\log|z|+C+o(1), \quad |z|\to 0
\]
and  $(2\pi)^{-1}\Delta (\Gamma *\theta)=\theta$.
\end{proof}

In contrast to the \cite{cm}, these boundary conditions give a counter rotating velocity as $|x|\rightarrow\infty$ with a linear transition
within the patch $E_0$. Additionally, this Biot-Savart law does not produce velocity in $L^2(S)$. The kernel $k(x,y)$ can be decomposed into
two pieces, one in $L^1(S)$ and the other which is bounded. Observe that
\begin{equation}\label{ubd}
 u(x,y,t) = \int_S \frac{\left((-\sin(y-\xi_2),\sgn(x-\xi_1)(\cos(y-\xi_2)-e^{-|x-\xi_1|})\right)}{2(\cosh(x-\xi_1)-
 \cos(y-\xi_2))}\theta(\xi)d\xi+ \frac{1}{2}\int_S (0,\sgn(x-\xi_1))\theta(\xi)d\xi.
\end{equation}
Notice that the first term is convolution with a kernel in $L^1(S)$,
as near $0$, it behaves like $(x^2+y^2)^{-1/2}$ and away from $0$ it
decays exponentially fast. The second term, however, does not decay
at all so the total kinetic energy of this problem could be
infinite. However, we can define the regularized energy $F(\theta)$
for the equation \eqref{ecyl} and compactly supported
$\theta(z,t)\in L^\infty(S)$ in the following way:
 \begin{equation}\label{rege}
F(\theta)(t) = \int_S\int_S \theta(z,t)\theta(\xi,t)
\log(\cosh(x_1-x_2)-\cos(y_1-y_2))dzd\xi, z=(x_1,y_1),\xi=(x_2,y_2).
 \end{equation}
When we consider the evolution of a patch under this flow, we will
show in section 2 that the total mass, the first coordinate of
center of mass, and regularized energy are conserved.\smallskip

We will use the following notation in the paper. If $A$ and $B$ are
sets, $A\Delta B$ stands for the symmetric difference. The symbol
$C$ will denote an absolute constant, and its actual value can
change from formula to formula. If $f_{1(2)}(x)$ are two positive
functions and
\[
\sup_{x}\frac{f_1(x)}{f_2(x)}<\infty
\]
we will write $f_1\lesssim f_2$. This is equivalent to writing
$f_1=O(f_2)$.\smallskip

We can now state our main Theorem. Similar to \cite{sv}, we will
show that the steady patch solution $E_0 = [-L,L]\times\mathbb{T}$
for sufficiently large $L$ is stable for all times. It is convenient
for our calculations to introduce what we will call a point of
centering.
\begin{definition}
 A \textit{point of centering} $x_c(t)$ for a patch $E(t)$ is the value in $\R$ so that
\[
 \int_{[x_c,\infty)\times\mathbb{T}}\chi_{E(t)}dz =
 \int_{(-\infty,x_c]\times\mathbb{T}}\chi_{E(t)}dz.
\]
\end{definition}
\noindent Notice that this point is not necessarily unique and the
set of all such points is always a segment or a single point. We use
a point of centering to make the comparison between the evolved
patch $E(t)$ and the simple rectangle $E_0$ more natural.

Our main result is the following Theorem.
\begin{theorem}\label{main}
There is an absolute constant $L_0>2$  such that the following
statement is true. If
\begin{enumerate}
\item[(a)]
$L>L_0$, $\epsilon<1$,
\item[(b)]
 $E$ is a compact subset of $S$ and $0$ is one of its
points of centering,
\item[(c)]
 $|E|=4L\pi$,
 \item[(d)] the regularized energy satisfies
\begin{equation}\label{sma}
F(\chi_E)=F(\chi_{E_0})+O(L\epsilon^2),
\end{equation}
\item[(e)] function $\theta$ solves 2D Euler equation \eqref{ecyl} with the Biot-Savart law given
by \eqref{strcyl} and \eqref{bscyl},
\end{enumerate} then
$\theta(t)=\chi_{E(t)}$ and $E(t)$ satisfies
\begin{eqnarray}
\int_S ||x-x_c(t)|-L|\chi_{E(t)\, \Delta\, E_0(t)}dxdy\lesssim
\epsilon^2,\\ |x_c(t)|\lesssim L^{-1}\epsilon^2 \label{pce1}
\end{eqnarray}
for all $t>0$.  Above $x_c(t)$ is any point of centering for $E(t)$
and $E_0(t)=[x_c(t)-L,x_c(t)+L]\times \mathbb{T}$.

Moreover, if  $\mu>\epsilon$, then
\[
|(E(t)\,\Delta\, E_0(t))\cap \{||x-x_c(t)|-L|>\mu\}|\lesssim
\epsilon^2\mu^{-1}.
\]
\end{theorem}
This result has a similar structure to the result in \cite{sv} for
circular patches in the plane, but with a point of centering
$x_c(t)$ in the role of the center of mass. However, that proof
relies on conserved quantities that do not hold in the cylindrical
case, namely $\int_{\mathbb{R}^2}|z|^2\theta(z)dz$. Instead, our
argument uses the one dimensional nature of the cylindrical problem
and the conservation of regularized energy. In the next section, we
will establish the necessary conserved quantities. In the third
section, we prove the main result on the stability.

\section{Preliminaries}
To proceed with our arguments on stability, we need a result
equivalent to Yudovi{\v{c}}'s result for the evolution of
$L^1\bigcap L^\infty$ solution \cite{yu,bm}. We are working on an
unbounded domain with periodicity in one direction. If we consider
the periodic extension of our problem to $\R^2$, we are interested
in bounded solutions on $\R^2$ with some decay in one direction.
Recent work by Kelliher and collaborators gives existence and
uniqueness for solutions to the Euler equations on $\R^2$ for
velocity $u$ and associated vorticity $\theta = \nabla\times u$
(defined in the sense of distributions) both bounded, with no decay
requirement. These results also include an adaptation of the
standard Biot-Savart law on $\R^2$ to relate $u$ and $\theta$
despite the lack of convergence of the standard integral identities.

We will apply this work in Appendix \ref{appe} to show the following
Theorem:
\begin{theorem}\label{eandu}
 Let $\theta_0(z)$ for $z\in S= \R\times\mathbb{T}$ be in $L^\infty (S)$ with compact support. Then there exists unique $(u,\theta)$
 with $u\in L^\infty(S)$ and $\theta\in  L^\infty(S)$ with compact support such that $\partial_t\theta + u\cdot\nabla\theta =0$ in the
 sense of distributions with \[u(z,t)=\nabla^\perp  (\Gamma*\theta) =
 \int_S
 \frac{(-\sin(y-\xi_2),\sinh(x-\xi_1))}{2(\cosh(x-\xi_1)-\cos(y-\xi_2))}\theta(\xi,t)d\xi\] and $\theta(z,0) =\theta_0(z)$.
\end{theorem}
We postpone the proof of this Theorem until the Appendix.

Once we have established existence and uniqueness, we can study the conserved quantities of the equation.
\begin{proposition}
 For $\theta(z,t)$ a solution to \eqref{ecyl}, the following quantities are conserved:
\begin{enumerate}
 \item the total mass $M = \int_S\theta(z)dz$,
\item the horizontal center of mass $x_0 = \int_S x \theta(z)dz$,
\item the total energy $F(\theta) = 2\int_S\int_S \theta(z)\theta(\xi) \Gamma(z-\xi)d\xi
dz$.
\end{enumerate}
\end{proposition}
\begin{proof}
The arguments for Theorem \ref{eandu} include that the vorticity $\theta$ is transported by the flow. Therefore, conservation of mass follows immediately.

Since we have $\partial_t\theta +u \cdot\nabla\theta=0$ only in the sense of distributions, we need to show carefully that we have conservation of
center of mass and regularized energy. Observe that for a smooth function $\varphi\in C^\infty([0,T],C^\infty_0(S))$, we have the following
representation:
\[
 \int_S \varphi(z,T)\theta(z,T)dz -  \int_S \varphi(z,0)\theta(z,0)dz = \int_0^T\int_S \left(\partial_t\varphi +u\cdot\nabla\varphi\right)\theta(z,t)dzdt.
\]
To show our other conserved quantities, we need to choose our smooth
bump function so that the desired quantity appears on the right hand
side of the expression above.

From \eqref{ubd}, we can bound the velocity $u$ by
\[
 \|u(t)\|_{L^\infty(S)}\le \|k_1\|_{L^1(S)}\|\theta(t)\|_{L^\infty(S)} +
 \|\theta(t)\|_{L^1(S)},
\]
where $k_1(z) = k(z) - (0,\sgn(x))$. The $L^1$ bound on $k$ is independent of time, and both $\|\theta\|_{L^1(S)}$ and $\|\theta\|_{L^\infty(S)}$
are conserved in time. Therefore,
 \begin{equation}
\|u(t)\|_{L^\infty(S)} \lesssim \|\theta(0)\|_{L^\infty(S)} +\|\theta(0)\|_{L^1(S)},
 \end{equation}
 and we know that up to a time $T$, $\theta$ is compactly supported. Let $b(x)$ be a smooth bump in the $x$-direction so that $b\equiv 1$ for
 every $x$ in the set $ [-R,R]$ where  $R$ is chosen so that $\mbox{supp}(\theta(z,t))\subseteq [-R,R]$ for all $t\in[0,T]$.

To see conservation of the  center of mass, let $\varphi(z,t) =
xb(x)$. Then,
\begin{align*}
 \int_S x\theta(z,T)dz -  \int_S x\theta(z,0)dz &= \int_0^T\int_S \left(\partial_t\varphi +u\cdot\nabla\varphi\right)\theta(z,t)dzdt\\
&= \int_0^T\int_S u_1(z,t)\partial_1(xb(x))\theta(z,t)dzdt\\
& = \int_0^T\int_S u_1(z,t)\theta(z,t)dzdt + \int_0^T\int_S x u_1(z,t) b'(x)\theta(z,t)dzdt.\\
\end{align*}
The second term is clearly $0$, as $b'(x) = 0$ on the support of $\theta(z,t)$. The first term can be rewritten as
\[
 \int_0^T\int_S \int_S \frac{-\sin(y_1-y_2)}{\cosh(x_1-x_2)-\cos(y_1-y_2)}\theta(\xi,t)\theta(z,t)d\xi dz
\]
which is also $0$, as the kernel is odd and rapidly decaying.

To see the conservation of regularized energy, we repeat a similar argument with $\varphi(z,t) = b(x)\Psi_\e(z,t)$ where
\[
 \Psi_\e(z,t) = \rho_\e *\int_S \log (\cosh(x-\xi_1)-\cos(y-\xi_2))(\rho_\e*\theta)d\xi.
\]
for a smooth, compactly supported bump $\rho_\e (z,t)$ on $[0,T]\times S$ defined as follows. Let $r(x)$ be
 a smooth bump supported in $[-1,1]$ on $\R$ with $\int r =1$. Let $r_\e(x) = \e^{-1}r(x/\e)$, let $r_{1,\e}(y)$
 be the periodic extension of $r_\e$ on $[-\pi,\pi]$, and let $r_{2,\e}(t) = r_\e(t)$. Then,
 $\rho_\e(x,y,t) = r_\e(x)r_{1,\e}(y)r_{2,\e}(t)$. The calculations involve changing the order of
 integration to rearrange the convolution with the mollifier $\rho_\e$ but are otherwise straightforward.
\end{proof}\bigskip

\section{Main Results}

We recall that $S=\mathbb{R}\times \mathbb{T}$. We first consider a
one-dimensional variational problem which will be important later.
Suppose $J$ is a measurable subset of $\mathbb{R}$ and $|J|=2L$.
Assume that $J$ is centered around the origin, such that $|J^+|=L,\,
J^+=J\cap [0,\infty)$. $\left(\right.$Notice here again that the
``centering points" for any set form a closed interval, which can
degenerate to a point$\left. \right)$. Consider a functional
\[
\Phi(\chi_J)=\int_{J}\int_{J} |x_1-x_2|dx_1dx_2.
\]
The following elementary Lemma holds true.
\begin{lemma}\label{one-d} We have
\[
\Phi(\chi_J)\ge  \Phi(\chi_{J_0})+CL\int_{J\,\Delta\,J_0} ||x|-L|dx,
\]
where $J_0=[-L,L]$.
%JB edit 12/24
\end{lemma}
\begin{proof}
Notice that this estimate is scale-invariant in $L$ and the actual
value of $L$ is not important. It is sufficient to assume that
$J^+=\cup_{j=1}^n I_j$ where $I_j$ are disjoint intervals (placed in
the order from left to right). Denote the gaps between them by
$\{Q_j\}$ so we have
\[
\mathbb{R}^+=Q_1\cup I_1\cup Q_2\cup I_2\cup\ldots Q_n\cup I_n\cup
[a,\infty).
\]
We can allow some gaps to be empty if necessary. The proof will
proceed as follows. We will close all gaps $\{Q_j\}$ and estimate
the total change in $\Phi$ that we denote $\delta \Phi$.

Let $J^{(1)}$ be the configuration obtained by closing the $Q_n$ gap
(sliding $I_n$ to the left) and denote the moved interval by
$I_n^{(1)}$ (e.g., $|I_n^{(1)}|=|I_n|$, $I_{n-1}$ and $I_n^{(1)}$
are adjacent to each other in $J^{(1)}$). Consider $J'=J\backslash
I_n=J^{(1)}\backslash I_n^{(1)}$. We have
\[
\Phi(\chi_J)=\Phi(\chi_{J'})+\Phi(\chi_{I_n})+2\int_{J'}d\xi
\int_{I_n} (x-\xi)dx,\quad
\Phi(\chi_{J^{(1)}})=\Phi(\chi_{J'})+\Phi(\chi_{I_n^{(1)}})+2\int_{J'}d\xi
\int_{I_n^{(1)}} (x-\xi)dx
\]
and
\[
\Phi(\chi_J)-\Phi(\chi_{J^{(1)}})= 2|Q_n|\int_{I_n}dx\int_{J'}
d\xi\ge 2L|Q_n| |I_n|,
\]
where the last inequality follows from $|J'|\ge |J^-|=L$, $J^-=J\cap
(-\infty,0]$.

 Compute inductively the total change $\delta_1 \Phi$ in $\Phi$ obtained by
closing all gaps $\{Q_j\}, j=k,\ldots,n$ to the right of $L$ (we
close them in the following order: $Q_n,Q_{n-1},\ldots,Q_k$):
\[
\delta_1 \Phi\ge 2L(|I_n||Q_n|+(|I_n|+|I_{n-1}|)|Q_{n-1}|+\ldots
+(|I_n|+\ldots+|I_k|)|Q_k|)=
\]
\begin{equation}\label{tv}
2L(|I_n|(|Q_n|+\ldots+|Q_k|)+|I_{n-1}|(|Q_{n-1}|+\ldots+|Q_k|)+\ldots+|I_k||Q_k|).
\end{equation}
Consider $[L,\infty)\cap J$ and denote $\epsilon=|[L,\infty)\cap
J|$. Let us divide all intervals $I_k,\ldots,I_n$ into two groups:
those that belong to the interval $[L,L+2\epsilon]$:
$\{I_k,\ldots,I_{j-1}\}$ and those that are to the right of
$L+2\epsilon$: $\{I_j,\ldots,I_n\}$. If $L+2\epsilon$ is an interior
point of some interval, we split this interval into two by creating
an empty gap at point $L+2\epsilon$. Notice that if $I_l\subset
[L+2\epsilon,\infty)$ (i.e., $I_l$ is in the second group), then
\[
|Q_l|+\ldots+|Q_k|\ge {\rm dist}(I_l, L)-\epsilon\ge \,\frac{{\rm
dist}(I_l, L)}{2}.
\]
Therefore, the contribution to \eqref{tv} coming from the second
group of intervals is bounded below by \[ L\sum_{l=j}^n {\rm
dist}(I_l,L)|I_l|\ge 0.5L\int_{x>L+2\epsilon} (x-L)\chi_Jdx.\]
$\Bigl(\Bigr.$Indeed, to see that last inequality we write
$I_l=[a_l,b_l]$ and notice that
\[
{\rm dist}(I_l, L)=a_l-L= (b_l-L)-(b_l-a_l)\ge b_l-L-\epsilon\ge
(b_l-L)/2\Bigl..\Bigr)
\]
Thus,
\[
\delta_1\Phi\gtrsim L\int_{x>L+2\epsilon}
(x-L)\chi_Jdx=L\int_{x>L+2\epsilon} (x-L)\chi_{J\Delta J_0}dx.
\]
Now that all gaps to the right of $L$ are closed, we call the new
configuration $\widehat J$, intervals are $\{\widehat I_j\}$, and
the gaps are $\{\widehat Q_j\}, j=1,\ldots m$. For the rightmost
interval $\widehat I_m$ we have  $\widehat I_m=[L,L+\epsilon]$ per
our construction and thus
\begin{equation}\label{zv1}
|\widehat I_m|=\epsilon.
\end{equation}
We are now closing all gaps in $[0,\infty)$ and estimating
$\delta_2\Phi$, the change of $\Phi$, from below. Notice first that
\begin{equation}\label{zv2}
\sum_{j=1}^m |\widehat Q_j|= \epsilon
\end{equation}
and
\begin{equation}\label{tv2}
\delta_2\Phi\ge 2L(|\widehat Q_1|(|\widehat I_1|+\ldots+|\widehat
I_{m}|)+\ldots+|\widehat Q_m| \cdot|\widehat I_m|).
\end{equation}
Observe that $(\widehat J\,\Delta\,J_0)\cap [0,L]$ is the union of
the $\{\widehat Q_l\}$. We now split all gaps $\{\widehat Q_l\}$
into two groups: those that belong to $[0,L-2\epsilon]$ (it could be
empty if, e.g., $\epsilon>L/2$) and all others. Notice that for each
gap $\widehat Q_p$ in the first group, we have
\[\sum_{j=p}^{m-1} |\widehat I_j| \ge  {\rm dist} (\widehat Q_p, L)-\epsilon\ge \frac{{\rm dist} (\widehat Q_p, L)}{2}. \]
Therefore, the contribution to \eqref{tv2} coming from the first
group of gaps is at least
\[
0.5 L\int_{0<x<L-2\epsilon}(L-x)\chi_{J\,\Delta\,J_0}dx.
\]
Collecting separately the terms in the right hand side of
\eqref{tv2} that contain $\widehat I_m$, we get
\[
2L|\widehat I_m|\cdot(|\widehat Q_1|+\ldots+|\widehat Q_m|)\ge
2L\epsilon^2
\]
by \eqref{zv1} and \eqref{zv2}. Keeping only the gaps in the first
group gives us
\[
\delta_2\Phi\geq \left(\sum_{{\rm first\, group\, of\,gaps }}
2L|\widehat Q_p|\sum_{j=p}^{m-1}|\widehat I_j|\right)+2L|\widehat
I_m|\cdot(|\widehat Q_1|+\ldots+|\widehat Q_m|)\gtrsim
L\int\limits_{x>0, 0<x<L-2\epsilon}|x-L|\chi_{J\Delta
J_0}dx+L\epsilon^2.
\]
We combine now the obtained inequalities to estimate the total
variation $\delta \Phi=\delta_1\Phi+\delta_2\Phi$ as
\[
\delta\Phi\gtrsim
L\int_{x>0,|x-L|>2\epsilon}|x-L|\chi_{J\,\Delta\,J_0}dx+L\epsilon^2\gtrsim
L\int_{x>0}|x-L|\chi_{J\,\Delta\,J_0}dx.
\]
Arguing in the same way for the half-line $\mathbb{R}^-$, we get the
statement of the Lemma since the resulting configuration after
closing all gaps is $J_0$.
\end{proof}

%begin JB edit 1/12 for transitions and notation
Our first goal is to control the regularized energy functional associated to the
Euler equation on $S$ defined in \eqref{rege}. Given an arbitrary vortex patch $E$, we will
need to transition to the vertical average of the patch to control a
portion of the energy. To that end, we define following functional
\begin{equation}\label{nfn}
\Phi(\rho)=\int_{\mathbb{R}}\int_{\mathbb{R}}
|x_1-x_2|\rho(x_1)\rho(x_2)dx_1dx_2.
\end{equation}
Notice first that $\Phi(\rho)<\infty$ implies
\[
\int_{\mathbb{R}} |x|\rho(x)dx<\infty.
\]

Assume that $\rho_j^+,\rho_j^-\in [0,1], j=0,1,\ldots$ are chosen
such that
\[
\sum_{j=0}^\infty\rho_j^{\pm}=1, \quad \sum_{j=1}^\infty
j\rho_j^{\pm}<\infty.
\]
Fix $\delta>0$ and consider the following convex set: $\mathcal{O}$
is the set of functions $\rho$, defined on $\mathbb{R}$, measurable,
and such that $0\le \rho\le 1$ and
\[
\int_{j\delta}^{(j+1)\delta}\rho(x)dx=\rho_j^+, \quad
\int_{-(j+1)\delta}^{-j\delta}\rho(x)dx=\rho_j^-, \quad j=0,1,\ldots
\]

In the next Lemma, we will study the following variational problem
\begin{equation}\label{vari}
\inf_{\rho\in \mathcal{O}}\Phi(\rho).
\end{equation}

In the Lemma \ref{triv1} from the Appendix, we prove that a
minimizer $\rho^*$ exists.

\begin{lemma}
 If $\rho^*$ is a minimizer
then $\rho^*$ is a characteristic function.
\end{lemma}
\begin{proof}
Notice that if $\rho_0$ and $\rho_1$ belong to $\mathcal{O}$, then
$\rho_t=t\rho_1+(1-t)\rho_0\in \mathcal{O},\, t\in (0,1)$ and
\begin{equation}\label{re}
\Phi''(\rho_t)=\int\int
|x_1-x_2|\delta(x_1)\delta(x_2)dx_1dx_2,\quad \delta=\rho_1-\rho_0.
\end{equation}
Going on the Fourier side, we have
\begin{equation}\label{rre}
\Phi''=-\int |\widehat\delta(k)|^2\frac{dk}{2k^2}<0\,,
\end{equation}
where the last integral makes sense since
$\int_{\mathbb{R}}\delta(x)dx=0, \int_{\mathbb{R}}
|x||\delta(x)|dx<\infty $ and so $\widehat\delta(0)=0,
(\widehat\delta)'\in L^\infty(\mathbb{R})$. That shows concavity of
the function in $t$. Now suppose that $\rho^*$ is not a
characteristic function, e.g., there is $\Sigma\subset I_j$ for some
$I_j$, such that $\epsilon<\rho^*<1-\epsilon$ on $\Sigma$ and
$|\Sigma|>0$ for some positive $\epsilon$. Then one can find
measurable function $\nu$ supported on $\Sigma$ such that
$\int_\mathbb{R}\nu dx=0$ and $\nu_t=\rho^*+t\nu\in \mathcal{O}$ for
$t\in (-\delta_1,\delta_1)$ with some small positive $\delta_1$.
However, the function $\Phi(\nu_t)$ is concave and $t=0$ can not be
a local minimum.
\end{proof}
%end JB edits 1/12

Two previous Lemmas imply
\begin{lemma} If $\rho$ is measurable, $0\le \rho\le 1$, and
\[
\int_0^\infty  \rho dx=L, \int_{-\infty}^0  \rho dx=L,
\]
then
%sd edit
\begin{equation}\label{raz1}
\Phi(\rho)\ge  \Phi(\chi_{J_0})+CL\int_{\mathbb{R}} ||x|-L|\cdot
|\rho(x)-\chi_{J_0}|dx.
\end{equation}
%sd edit ends
\end{lemma}
\begin{proof}
Indeed, consider the grid $\{j\delta\}, j\in \mathbb{Z}$ with step
$\delta$ and let
\[
\rho_j^+=\int_{j\delta}^{(j+1)\delta} \rho dx, \quad
\rho_j^-=\int_{-(j+1)\delta}^{-j\delta} \rho dx, \quad j=0,1,\ldots
\]
The variational argument given above shows that the value of $\Phi$
will decrease if we replace $\rho$ by a minimizer which needs to be
a characteristic function $\chi_{J}$. By Lemma \ref{one-d}, we  get
\[\Phi(\rho)\ge  \Phi(\chi_{J})\ge \Phi(\chi_{J_0})+CL\int_{J\,\Delta\,J_0}  ||x|-L|dx
=\Phi(\chi_{J_0})+CL\int_{\mathbb{R}} ||x|-L| \cdot
|\chi_{J}-\chi_{J_0}|           dx.\] Sending $\delta\to 0$ and
using
\[
\lim_{\delta\to0} \int_{\mathbb{R}} ||x|-L| \cdot
|\chi_{J}-\chi_{J_0}| dx=\int_{\mathbb{R}} ||x|-L|\cdot
|\rho(x)-\chi_{J_0}|dx,
\]
we get the statement of the Lemma.
%end JB edit
\end{proof}

%Begin JB edits to correct with residue calculation
%Begin JB edits 1/12 for notation
Now we turn our attention to the full energy functional for 2D Euler in $S$. Let
\[
F(\chi_E) = \int_S\int_S \chi_E(z)\chi_E(\xi)
\log(\cosh(x_1-x_2)-\cos(y_1-y_2))dzd\xi, \quad
z=(x_1,y_1),\,\xi=(x_2,y_2).
 \]
Observe that
\begin{equation}\label{repr1}
F(\chi_E)=(2\pi)^2\int_{\mathbb{R}}\int_{\mathbb{R}}
\rho_E(x_1)\rho_E(x_2)|x_1-x_2|dx_1dx_2+F_1(\chi_E)
-\log(2)\|\chi_E\|_{L^1(S)}^2\,,
\end{equation}
where
\[
\rho_E(x)=\frac{1}{2\pi}\int_{-\pi}^\pi \chi_E(x,y)dy
\]
and
\[
F_1(\chi_E)=\int_S\int_S \chi_E(z)\chi_E(\xi)
\log(\cosh(x_1-x_2)-\cos(y_1-y_2))dzd\xi-(2\pi)^2\Phi(\rho_E)+\log(2)\|\chi_E\|_{L^1(S)}^2\,.\]
Here $\Phi$ is defined in \eqref{nfn}.

Assume that $E\subset S$, $|E|=4\pi L$, and $|E\cap \{x>0\}|=2\pi
L$, i.e., $E$ is centered around $0$.

\begin{theorem}\label{dif1}
% sd edit begins
There is $L_0>2$ such that for every $L>L_0$ we have
\[
|F_1(\chi_E)-F_1(\chi_{E_0})|\lesssim \int_S ||x|-L|
\chi_{E\,\Delta\, E_0}dxdy\,,
\]
% sd edit ends
where $E_0=[-L,L]\times \mathbb{T}$.
\end{theorem}
\begin{proof}
Consider $f\in L^1(S)\cap L^\infty(S)$ and let $f_0=\chi_{E_0}$. If
$f=f_0+h$, we have
\[
F_1(f)=\int_S\int_S K(z,\xi)(f_0(z)+h(z))(f_0(\xi)+h(\xi))dzd\xi
\]
and \begin{equation}\label{yadro}
K(z,\xi)=\log(\cosh(x_1-x_2)-\cos(y_1-y_2))-|x_1-x_2| +\log 2.
\end{equation}
 Notice
that $K$ is symmetric and
\begin{equation}\label{ala}
\int_S K(z,\xi)f_0(\xi)d\xi=0.
\end{equation}
Indeed,
\[
\int_{-\pi}^\pi\log(\cosh(x_1-x_2)-\cos(y_1-y_2))dy_2=\int_{-\pi}^\pi
\log \left(\frac{\kappa^2+1}{2\kappa}-\cos y\right)dy\,,
\]
where  $\kappa\geq 1$ solves equation
\[
\frac{\kappa^2+1}{2\kappa}=\cosh(x_1-x_2).
\]
Clearly, $\kappa=e^{|x_1-x_2|}$. We continue as
\[
\int_{-\pi}^\pi \log \left(\frac{\kappa^2+1}{2\kappa}-\cos
y\right)dy=2\int_{-\pi}^\pi\log|\kappa-e^{iy}|dy-\int_{-\pi}^\pi\log(2\kappa)dy.
\]
Function $\log|\kappa-z|$ is harmonic in $z$ in the unit disc and
the mean-value theorem for harmonic functions gives
\[2\int_{-\pi}^\pi\log|\kappa-e^{iy}|dy-\int_{-\pi}^\pi\log(2\kappa)dy=
2\pi\log\kappa-2\pi\log 2= 2\pi|x_1-x_2|-2\pi \log 2 \,.
\] %End JB Edit 12/24, 1/12
Now, \eqref{ala} easily follows.

 Then,
\[F_1(\chi_E)-F_1(\chi_{E_0}) = \int_S\int_S K(z,\xi)h(z)h(\xi)dzd\xi
\]
with $h=\chi_E-\chi_{E_0}$. Notice that
$|h|=|\chi_E-\chi_{E_0}|=\chi_{E\,\Delta \,E_0}$ and let
% sd edit begins
$A=E\,\Delta\,E_0$. So, it is sufficient to control $\int_A\int_A
|K(z,\xi)|dzd\xi$ to complete the proof. We turn now to the kernel,
$K(z,\xi)$. The following is immediate: $K(z,\xi)$ is translation
invariant, i.e., $K(z,\xi)=K(z-\xi,0)$, and also
\begin{equation}\label{es33}
|K(z,0)|\lesssim \left\{
\begin{array}{cc}%begin JB edit 1/4; possibly not necessary, but the easiest way for me to work out the details.
e^{-0.1|z|}, & |z|>1\\
1+|\log|z||,& |z|<1
\end{array}\quad %end JB edit 1/4
\right.
\end{equation}
for every $z\in S$.
 We have the following trivial bound
\begin{equation}\label{b22}
\int_A\int_A |K(z,\xi)|dzd\xi\lesssim |A|\cdot \min(1,|A|\cdot
(1+|\log |A||)),
\end{equation}
where we took into account two estimates:
\begin{equation}\label{b00}
\int_S |K(z,0)|dz<C
\end{equation}
and
\begin{equation}\label{b23}
\quad \left|\int_{A}\min\{1,1+|\log|\xi||\}d\xi\right|\lesssim
|A|\cdot |\log|A||,{\,\,\rm provided\, that}\, \, |A|<0.5.
\end{equation}
The last bound follows from the observation that the maximizer for
that integral is the ball centered at the origin.\smallskip

 Given $A$, we have two cases:
% sd edit ends

1. $|A|>1$. Then
\[
\int_A\int_A |K(z,\xi)|dzd\xi\lesssim  |A|\lesssim \int_{A}
||x|-L|dxdy\,,
\]
where the last inequality holds for all  $A\subset S$ satisfying
$|A|\ge 1$ (see, e.g., Lemma \ref{petal} from Appendix C).

2. $|A|\le 1$. Then, we write $A=A_1\cup A_2$ where $A_1=A\cap
\{||x|-L|<1\}$.\smallskip

Consider $A_1$. If $A_1=A_1^+\cup A_1^-, \,A_1^+=A_1\cap
\{|x-L|<1\},\, A_1^-=A_1\cap \{|x+L|<1\}  $, then
\[
\int_{A_1}\int_{A_1}|K(z,\xi)|dzd\xi\lesssim \int_{A_1} ||x|-L|dxdy
\]
by Lemma \ref{sm1} and Lemma \ref{lg} that are proved in Appendix C.
The remaining terms are then bounded using \eqref{b00} as follows
\[
\int_{A_2}\int_{A_2}|K(z,\xi)|dzd\xi\lesssim |A_2|\lesssim
\int_{A_2} ||x|-L|dxdy
\]
and similarly
\[
\int_{A_1}\int_{A_2} |K(z,\xi)|dzd\xi\lesssim |A_2|\lesssim
\int_{A_2}||x|-L|dxdy.
\]
The proof of Theorem \ref{dif1} is now completed.
\end{proof}
We define $\cal{M}$ as the collection of measurable sets $E\subset
S$ such that $E$ is centered around the origin, $|E|=4\pi L$.
Theorem \ref{dif1} along with \eqref{raz1} and \eqref{repr1} give
\begin{corollary}\label{argun}
% sd edit begin
There is $L_0>2$ such that for every $L>L_0$ and  $E\in \cal{M}$, we
have
\[
F(\chi_E)\ge F(\chi_{E_0})+CL\int_{E\,\Delta\,E_0} ||x|-L|dxdy.
\]
Therefore, $\arg\min\limits_{E\in \cal{M}}F(\chi_E)=E_0=[-L,L]\times
\mathbb{T}$.
% sd edit ends
\end{corollary}

Now we are ready to apply our estimates to the Euler dynamics. We
will start by proving the following Theorem which is identical to
Theorem \ref{main} except that the estimate \eqref{pce1} for points
of centering is missing.
\begin{theorem}\label{m3}

There is an absolute constant $L_0>2$  such that the following
statement is true. If
\begin{enumerate}
\item[(a)]
$L>L_0$, $\epsilon<1$,
\item[(b)]
 $E$ is a compact subset of $S$ and $0$ is one of its
points of centering,
\item[(c)]
 $|E|=4L\pi$,
 \item[(d)] the regularized energy satisfies
\begin{equation}\label{sma1}
F(\chi_E)=F(\chi_{E_0})+O(L\epsilon^2),
\end{equation}
\item[(e)] function $\theta$ solves 2D Euler equation \eqref{ecyl} with the Biot-Savart law given
by \eqref{strcyl} and \eqref{bscyl},
\end{enumerate} then
$\theta(t)=\chi_{E(t)}$ and $E(t)$ satisfies
\begin{eqnarray}\label{lim1}
\int_S ||x-x_c(t)|-L|\chi_{E(t)\, \Delta\, E_0(t)}dxdy\lesssim
\epsilon^2
\end{eqnarray}
for all $t>0$.  Above $x_c(t)$ is any point of centering for $E(t)$
and $E_0(t)=[x_c(t)-L,x_c(t)+L]\times \mathbb{T}$.

Moreover, if  $\mu>\epsilon$, then
\[
|(E(t)\,\Delta\, E_0(t))\cap \{||x-x_c(t)|-L|>\mu\}|\lesssim
\epsilon^2\mu^{-1}.
\]
\end{theorem}
\begin{proof}
Notice that %sd edit begins
 $F(\chi_{E(t)})$ and $|E(t)|$ are invariants.
Therefore, Corollary \ref{argun} gives
\[
L\int_S ||x-x_c(t)|-L|\chi_{E(t)\, \Delta\, E_0(t)}dxdy\lesssim
L\epsilon^2
\]
and the statements follow.
\end{proof}

The following Lemma gives a simple geometric condition for
\eqref{sma1} to hold.

\begin{lemma}If $E$ is centered around
the origin, $|E|=4L\pi$, and $\{|x|<L-\epsilon\}\subseteq E\subseteq
\{|x|<L+\epsilon\}$, then
\begin{equation}\label{sma}
F(\chi_E)=F(\chi_{E_0})+O(L\epsilon^2).
\end{equation}
\end{lemma}
\begin{proof} Consider the representation \eqref{repr1} for $E$ and compare it to the same representation for $E_0$. For the
second term, we use Theorem \ref{dif1} to get
\[
|F_1(\chi_E)-F_1(\chi_{E_0})|\lesssim \int_S ||x|-L|
\chi_{E\,\Delta\, E_0}dxdy\lesssim \epsilon^2.
\]
If we write $\rho_E=\chi_{J_0}+\delta$, then the first term in
\eqref{repr1} gives
\[
\int_{\mathbb{R}}\int_{\mathbb{R}}
\rho_E(x_1)\rho_E(x_2)|x_1-x_2|dx_1dx_2=\int_{\mathbb{R}}\int_{\mathbb{R}}
(\chi_{J_0}(x_1)+\delta(x_1))(\chi_{J_0}(x_2)+\delta(x_2))|x_1-x_2|dx_1dx_2
\]
with $\|\delta\|_{L^\infty}\lesssim  1$, $\int_{\mathbb{R}^+} \delta
dx=\int_{\mathbb{R}^-} \delta dx  =0,$ and ${\rm
supp}\,\delta\subseteq \{L-\epsilon<|x|<L+\epsilon\}$. Notice that
(see, e.g., \eqref{re}, \eqref{rre})
\[
\int_{\mathbb{R}}\int_{\mathbb{R}}
\delta(x_1)\delta(x_2)|x_1-x_2|dx_1dx_2\leq 0.
\]
For the cross product,
\[
\int_{\mathbb{R}}\int_{\mathbb{R}}
\delta(x_1)\chi_{J_0}(x_2)|x_1-x_2|dx_1dx_2=\int_{-L-\epsilon}^{-L+\epsilon}
\delta(x_1)\left(\int_{-L}^L
|x_1-x_2|dx_2\right)dx_1+\int_{L-\epsilon}^{L+\epsilon}
\delta(x_1)\left(\int_{-L}^L |x_1-x_2|dx_2\right)dx_1\,.
\]
Consider, e.g., the first integral. We have
\[
||x_1-x_2|-|(-L)-x_2||\le |x_1+L|\le \epsilon
\]
and, therefore,
\[
\int_{-L-\epsilon}^{-L+\epsilon} \delta(x_1)\left(\int_{-L}^L
|x_1-x_2|dx_2\right)dx_1=O(L\epsilon^2)+\int_{-L-\epsilon}^{-L+\epsilon}
\delta(x_1)\left(\int_{-L}^L |(-L)-x_2|dx_2\right)dx_1\,.
\]
Similarly,
\[
\int_{L-\epsilon}^{L+\epsilon} \delta(x_1)\left(\int_{-L}^L
|x_1-x_2|dx_2\right)dx_1=O(L\epsilon^2)+\int_{L-\epsilon}^{L+\epsilon}
\delta(x_1)\left(\int_{-L}^L |L-x_2|dx_2\right)dx_1\,.
\]
However,
\[
\int_{-L-\epsilon}^{-L+\epsilon} \delta(x_1)\left(\int_{-L}^L
|(-L)-x_2|dx_2\right)dx_1=\int_{L-\epsilon}^{L+\epsilon}
\delta(x_1)\left(\int_{-L}^L |L-x_2|dx_2\right)dx_1=0\,,
\]
because $\int_{\mathbb{R}^{\pm}} \delta dx=0$ and we have the
statement of the Lemma.
\end{proof}

To complete the proof of Theorem \ref{main}, we are left with studying the dynamics of $x_c(t)$, the point of
centering.

\begin{lemma}\label{lims} In the previous Theorem \ref{m3}, a centering point $x_c(t)$ satisfies
\[
|x_c(t)|\lesssim L^{-1}\epsilon^2
\]
for all times.
\end{lemma}
\begin{proof}
 For the patch $\chi_E(t)$ we have also that the $x$-coordinate of the center of mass is conserved and equal to zero, so:
 \[\int_S (x-x_c(t))\chi_{E(t)} (z) dxdy = -4\pi L x_c(t).\] It suffices to bound the
 left hand side by $\epsilon^2$. Recall $E_0(t)=[x_c(t)-L,x_c(t)+L]\times \mathbb{T}$ observe that
 \[\int_S (x-x_c(t))\chi_{E(t)} (z) dxdy = \int_S (x-x_c(t))(\chi_{E(t)}(z)-\chi_{E_0(t)}(z)) dxdy. \]
We use the fact that $x_c(t)$ is the centering point for both $E(t)$
and $E_0(t)$ to write
\[
\int_{x>x_c(t)} (x-x_c(t))(\chi_{E(t)}(z)-\chi_{
E_0(t)}(z))dxdy=\int_{x>x_c(t)}
(x-x_c(t)-L)(\chi_{E(t)}(z)-\chi_{E_0(t)}(z)) dxdy\le
\]
\[\int_{x>x_c} |x-x_c(t)-L| \cdot |\chi_{E(t)}(z)-\chi_{E_0(t)}(z)| dxdy =\int_{x>x_c} |x-x_c(t)-L| \cdot \chi_{E(t)\Delta
E_0(t)}(z) dxdy \lesssim \epsilon^2\] as follows from \eqref{lim1}.
The integral over $x<x_c(t)$ is handled similarly. Thus,
 $ |x_c(t)|\lesssim L^{-1}\epsilon^2.$
\end{proof}

The Theorem \ref{m3} and Lemma \ref{lims} give the proof of Theorem
\ref{main}.

\bigskip

\begin{appendix}
\section{Existence and Uniqueness of Solution on $S$}\label{appe}

Now we will discuss the existence and uniqueness result stated in
Section 2.
\begin{theorem}
 Let $\theta_0(z)$ for $z\in S= \R\times\mathbb{T}$ be in $L^\infty (S)$ with compact support in $S$. Then
  there exists unique $(u,\theta)$ with $u\in L^\infty(S)$ and $\theta\in L^\infty(S)$ with compact support
  in $S$ such that $\partial_t\theta + u\cdot\nabla\theta =0$ in the sense of distributions with
  \[u(z,t)=\nabla^\perp  (\Gamma*\theta) = \int_S \frac{(-\sin(y-\xi_2),\sinh(x-\xi_1))}{2(\cosh(x-\xi_1)-\cos(y-\xi_2))}\theta(\xi,t)d\xi\]
  and $\theta(z,0) =\theta_0(z)$.
\end{theorem}

This Theorem is a corollary of the following result of Kelliher
\cite{ke}. Note that in that work, the space $\mathcal{S}(\R^2)$ is
the space of all divergence-free vector fields $u$ with vorticity
$\theta(u)$ so that
\[
\|u\|_{L^\infty} +\|\theta(u)\|_{L^\infty}<\infty.\]
The goal is to consider bounded velocity and vorticity without assumptions on their smoothness, so $\nabla\cdot u= 0$ and $\theta(u)=\nabla\times u$ in
 the sense of distributions.
Moreover, we say $u\in \mathcal{S}$ with vorticity $\theta$ is a bounded solution if
\begin{enumerate}
\item $\partial_t\theta +u\cdot\nabla\theta= 0$ in the sense of
distributions,
\item the vorticity is transported by the flow.
\end{enumerate}
Now we can state the Theorem from \cite{ke}:
\begin{theorem}[Theorem 2.9, \cite{ke}]\label{thmke}
 Assume that $u^0$ is in $\mathcal{S}(\R^2)$, let $T>0$ be arbitrary, and fix $U_\infty(t)\in(C[0,T])^2$ with $U_\infty(0)=0$. Let
 $\mathcal{K}(y) = \dfrac{y^\perp}{|y|^2}$. There exists a bounded solution $u$ to the Euler equations in $\R^2$, and this solution satisfies
a renormalized Biot-Savart law
\[u(t) -u^0 = U_\infty(t)+\lim_{R\rightarrow\infty} (a_R \mathcal{K})*(\theta(t)-\theta^0)\] on $[0,T]\times\R^2$ for all smooth, compactly supported,
 radial cutoff functions $a_R(x) = a(x/R)$  with $a(x) = 1$ for $|x|<1$ and $a(x)=0$ for $|x|>2$.
This solution is unique among all solutions $u$ with $u(0)=u^0$ that satisfy the given renormalized Biot-Savart law.
\end{theorem}
The proof of this Theorem is given in its entirety in \cite{ke}. The
vector field $U_\infty(t)$ allows the work to characterize the
non-uniqueness of solutions when $u$ is only bounded. We will use
this result to prove Theorem \ref{eandu}, and the choice of
$U_\infty(t)$ is naturally proscribed by the boundary conditions on
the stream function in \eqref{bscyl}.
\begin{proof}[Proof of Theorem \ref{eandu}]
 Let $\theta^0(x,y)$ on $\R^2$ be the periodic extension of $\theta_0(z)$ and define $u_0(z) = \nabla^\perp\Gamma*\theta_0(z)$ with periodic extension
 $u^0(x,y)$. Then by Theorem \ref{thmke}, there exist unique $u$ and $\theta$ defined on all of $\R^2$ so that
\[u(t) -u^0 = U_\infty(t)+\lim_{R\rightarrow\infty} (a_R \mathcal{K})*(\theta(t)-\theta^0)\] on $[0,T]\times\R^2$.
The solution $u(x,y,t)$ and $\theta(x,y,t)$ are periodic by uniqueness.

First we will show that renormalized Biot-Savart law is equivalent to the cylindrical Biot Savart law given by \eqref{strcyl} and \eqref{bscyl}.
 If we consider any $g(\xi)\in L^\infty(\mathbb{R}^2)$ such that $g(\xi)=g(\xi+(0,2\pi))$ and $g(\xi)=0$ for $|\xi_1|>R$, then
\[\int_{\R^2} a_R(z-\xi) \frac{(z-\xi)^\perp}{|z-\xi|^2}g(\xi)d\xi= \int_{\R}\sum_{k=-\infty}^\infty\int_{-\pi}^\pi a_R(z-\xi_k)
\frac{(z-\xi_k)^\perp}{|z-\xi_k|^2}g(\xi)d\xi\,,\] where $\xi_k =
\xi+(0,2\pi k)$ by the periodicity of $g$.  Observe that
\[\sum_{k=-\infty}^{\infty}\frac{(z-\xi_k)^\perp}{|z-\xi_k|^2} = \sum_{k=-\infty}^\infty \frac{(-(x_2-y_2-2\pi k), x_1-y_1)}{(x_1-y_1)^2 +(x_2 - y_2-2\pi k)^2}\] and by
 Poisson's Summation Formula, we have the following identity:
\begin{lemma}\label{poiss}
For $a,b\neq 0$
\[ \sum_{k=-\infty}^\infty \frac{(-(b-2\pi k), a)}{a^2 +(b-2\pi k)^2} = \frac{(-\sin(b), \sinh(a))}{2(\cosh(a)-\cos(b))}.\]
\end{lemma}
Since $a_R$ tends to $1$ uniformly and $\theta$ is compactly
supported in ${S}$, we can conclude that limit in $R$ converges to
the desired cylindrical Biot-Savart law.

The boundary conditions on $\Psi$ in \eqref{bscyl} require that
\[ \lim_{x\rightarrow +\infty}u_2(x,y)= -\lim_{x\rightarrow-\infty}u_2(x,y). \]
Since $k*\theta(x,y,t)$ satisfies this equality, it must be true
that $U_\infty^{(2)}(t) = -U_\infty^{(2)}(t)=0$. As for
$U^{(1)}_\infty(t)$, observe that if $\theta_1(x,y,t)$ solves the
cylindrical problem with horizontal drift
\[\partial_t\theta_1 + U_\infty^{(1)}(t)\partial_1\theta_1 +(k*\theta_1)\nabla \theta_1 = 0, \]
the translated function $\widetilde{\theta} = \theta_1(x+F(t),y,t)$ where $F(t) = \int_0^t U_\infty^{(1)}(s)ds$ satisfies
\[\partial_t\widetilde{\theta} + (k*\widetilde{\theta})\nabla\widetilde{\theta} = 0.\] By setting $U_\infty^{(1)}(t)\equiv 0$, we factor out
 the possibility of a moving reference frame in the horizontal direction.

It only remains to show that the Lemma \ref{poiss}  holds.
\begin{proof}[Proof of Lemma \ref{poiss}]
We can compute this sum explicitly. For the first component, Poisson
summation formula gives
\[
 \sum_{k\in \mathbb{Z}} \frac{-(b-2\pi k)}{a^2 +(b-2\pi k)^2} = -\frac{1}{2i}\sum_{n=1}^\infty
 (e^{ibn}e^{-|a|n}-e^{-ibn}e^{-|a|n})\,,
\]
since we have by residue calculus
\[
    \int_{-\infty}^\infty\frac{e^{2\pi i xn}2\pi x}{a^2+(2\pi x)^2}dx =\left\{\begin{array}{ll}
 \dfrac{ie^{-|a|n}}{2}, &\mbox{ for }n>0, \\
&\\
-\dfrac{ie^{|a|n}}{2}, &\mbox{ for }n<0, \\
&\\
0,&  n=0.
\end{array}\right.
   \]
Since $|e^{-|a|\pm i b}|\le e^{-|a|}<1$ for $|a|>0$, we have
\[
 \sum_{k\in \mathbb{Z}} \frac{-(b-2\pi k)}{a^2 +(b-2\pi k)^2} = -\frac{\sin(b)}{2(\cosh(a) - \cos(b))}.
\]
For the second component, assume  that $a>0$. Then,
\[
\sum_{k\in \mathbb{Z}} \frac{a}{a^2 +(b-2\pi k)^2} =
\frac{1}{2}+\frac{1}{2}\sum_{n=1}^\infty
e^{-(a+ib)n}+e^{-(a-ib)n}\,,
\]
since we have
\[
    \int_{-\infty}^\infty\frac{e^{2\pi i xn}}{a^2+(2\pi x)^2}dx =
    \frac{e^{-|an|}}{2|a|}.
   \]
As before, we can use the geometric series and see that
\[
\sum_{k\in \mathbb{Z}} \frac{a}{a^2 +(b-2\pi k)^2} =
\frac{1}{2}+\frac{1}{2}\left(\frac{e^{-(a+ib)}}{1-e^{-(a+ib)}}+\frac{e^{-(a-ib)}}{1-e^{-(a-ib)}}\right)
= \frac{\sinh(a)}{2(\cosh(a)-\cos(b))}.
\]
Since this expression is odd in $a$, it holds for $a<0$ as well.
\end{proof}
Finally, the fact that $\theta$ is periodic in the whole plane and
is transported by the flow gives\[ \|\theta(t)\|_{L^1(S)} =
 \|\theta_0\|_{L^1(S)} \mbox{ and } \|\theta(t)\|_{L^\infty(S)} = \|\theta_0\|_{L^\infty(S)}.                                                                                                                                                       \]
By the cylindrical Biot-Savart law, we know that $\|u\|_{L^\infty(S)}\lesssim \|\theta(0)\|_{L^1(S)}+\|\theta(0)\|_{L^\infty(S)}$
and $\theta$ will remain compactly supported.
\end{proof}

\section{Existence of Minimizer}

In this Appendix, we prove a standard result about existence of a
minimizer in the variational problem \eqref{vari}.

\begin{lemma}\label{triv1}
The problem \eqref{vari} has a minimizer.
\end{lemma}

\begin{proof}
Denote
\[
\sigma=\inf_{\rho\in\mathcal{O}} \Phi(\rho)
\]
and $\rho_n$ is a minimizing sequence: $\Phi(\rho_n)\to \sigma$.
Recall that
\begin{equation}\label{tight}
\sup_n\int (1+|x|)\rho_ndx<C.
\end{equation}
Consider $\{\rho_n\}$. We can choose a subsequence
$\{\rho_{k_n}\}\to \rho^*$ weakly over all compact sets in $\mathbb{R}$
and clearly $\rho^*\in \mathcal{O}$. Let us rename this
$\{\rho_{k_n}\}$ back as $\{\rho_n\}$ for convenience. We have
\[
\int_{-T}^T (1+|x|)\rho^*dx=\lim_{n\to\infty} \int_{-T}^T
(1+|x|)\rho_ndx\le \liminf_{n\to\infty} \int (1+|x|)\rho_ndx<C\,,
\]
so
\begin{equation}\label{tight1}
(1+|x|)\rho^* \in L^1(\mathbb{R}).
\end{equation}
because $T$ is arbitrary. Similarly, we conclude that
\begin{equation}\label{in}
\int_{0}^\infty x\rho^*dx\le \liminf_{n\to\infty} \int_0^\infty
x\rho_ndx, \quad \int_{-\infty}^0 |x|\rho^*dx\le
\liminf_{n\to\infty} \int_{-\infty}^0 |x|\rho_ndx.
\end{equation}
Notice also that
\[
\int_{\mathbb{R}^+}\rho^*dx=\int_{\mathbb{R}^-}\rho^*dx=1
\]
as follows from the definition $\mathcal{O}$. We will need the
following result
\begin{equation}\label{inw1}
\lim_{n\to\infty}\int_0^\infty x_1\rho_n(x_1)\int_{x_1}^\infty
\rho_n(x_2)dx_2dx_1=\int_0^\infty x_1\rho^*(x_1)\int_{x_1}^\infty
\rho^*(x_2)dx_2dx_1.
\end{equation}
It is due to the tightness estimate \eqref{tight}, \eqref{tight1},
weak convergence, and Dominated Convergence Theorem.

We now prove that $\rho^*$ is a minimizer, i.e., that
$\Phi(\rho^*)=\sigma$. Write $\Phi(\rho_n)$ as
\[
\Phi(\rho_n)=I_1+I_2=\int_0^\infty \rho_n(x_1)
dx_1\int_{-\infty}^\infty |x_2-x_1|\rho_n(x_2)dx_2+\int_{-\infty}^0
\rho_n(x_1) dx_1\int_{-\infty}^\infty |x_2-x_1|\rho_n(x_2)dx_2.
\]
For $I_1$, we have
\[
I_1=\int_0^\infty \rho_n(x_1) dx_1\int_0^\infty
|x_2-x_1|\rho_n(x_2)dx_2+\int_0^\infty \rho_n(x_1) dx_1\int_0^\infty
(x_2+x_1)\rho_n(-x_2)dx_2.
\]
Consider the first integral. By symmetry, it is equal to
\[
2\int_0^\infty \rho_n(x_1)dx_1\int_0^{x_1}
(x_1-x_2)\rho_n(x_2)dx_2=2\left(\int_0^\infty
x_1\rho_n(x_1)dx_1\right)\left(\int_0^\infty \rho_n(x_2)dx_2\right)-
\]
\[
2\int_0^\infty \rho_n(x_1)dx_1\int_0^{x_1} x_2\rho_n(x_2)dx_2-
2\int_0^\infty x_1\rho_n(x_1)dx_1\int_{x_1}^\infty \rho_n(x_2)dx_2.
\]
Notice that the last two terms are equal to each other and
\[
\int_0^\infty \rho^*dx=\int_0^\infty \rho_ndx=1.
\]
Thus, we are left with
\[
2\left(\int_0^\infty x_1\rho_n(x_1)dx_1\right)\left(\int_0^\infty
\rho^*(x_2)dx_2\right)-4\int_0^\infty
x_1\rho_n(x_1)dx_1\int_{x_1}^\infty \rho_n(x_2)dx_2=
\]
\[
2\left(\int_0^\infty x_1\rho_n(x_1)dx_1\right)\left(\int_0^\infty
\rho^*(x_2)dx_2\right)-4\int_0^\infty
x_1\rho^*(x_1)dx_1\int_{x_1}^\infty \rho^*(x_2)dx_2+o(1)
\]
by \eqref{inw1}. Then, \eqref{in} applied to the first term in the
last expression gives

\[
\int_0^\infty \rho^*(x_1)dx_1\int_0^{x_1}
(x_1-x_2)\rho^*(x_2)dx_2\le \liminf_{n\to\infty} \int_0^\infty
\rho_n(x_1)dx_1\int_0^{x_1} (x_1-x_2)\rho_n(x_2)dx_2.
\]
In a similar way, one shows that
\[
\int_0^\infty \int_0^\infty
\rho^*(x_1)\rho^*(-x_2)(x_1+x_2)dx_1dx_2\le \liminf_{n\to\infty}
\int_0^\infty \int_0^\infty
\rho_n(x_1)\rho_n(-x_2)(x_1+x_2)dx_1dx_2.
\]
The integral $I_2$ can be handled similarly. Adding up these
inequalities, we see that
$
\Phi(\rho^*)\le \sigma
$
so $\rho^*$ is a minimizer.

\end{proof}

\section{Three auxiliary Lemmas}

In this Appendix, we will prove three results used in the main text.
We introduce notation  $\Box=[-1,1]\times \mathbb{T}$.

\begin{lemma}\label{petal} We have
\[
\min_{Q\subset S, |Q|=s}\int_\Box |x|\chi_Qdxdy=\frac{s^2}{8\pi}
\]
and the minimum is achieved on $Q_{min}=[-s/4\pi,s/4\pi]\times
\mathbb{T}$.
\end{lemma}
\begin{proof}
The result follows immediately from the  structure of the weight
$|x|$ against which $\chi_Q$ is being integrated.
\end{proof}

Recall the kernel $K(z,\xi)$ which was introduced in \eqref{yadro}.

\begin{lemma}\label{sm1} Let $L>2$. If $A_1^+\subset \{|x-L|<1\}$ and
$A_1^-\subset \{|x+L|<1\}$, then
\begin{equation}\label{rh1}
\int_{A_1^+}\int_{A_1^-} |K(z,\xi)|dzd\xi\lesssim \int_{A_1^+\cup
A_1^-} ||x|-L|dxdy.
\end{equation}
\end{lemma}
\begin{proof}
The estimate \eqref{es33}  implies that
\[
\int_{A_1^+}\int_{A_1^-} |K(z,\xi)|dzd\xi\lesssim |A_1^+|\cdot
|A_1^-|\le 0.5(|A_1^+|^2+|A_1^-|^2).
\]
Now, to estimate $|A_1^{\pm}|^2$, we only need to change variables
as $x\pm L=\widehat x$ in the right-hand side of \eqref{rh1} and
notice that
\[
|B|^2\lesssim \int_S |\widehat x|\chi_{B}d\widehat z
\]
for every measurable set $B\subset \{|\widehat x|<1\}$ by Lemma
\ref{petal}.
\end{proof}

\begin{lemma}\label{lg} If $A$ is a measurable subset of\, $\Box$, then
\begin{equation}\label{pet}
\int_A\int_A |\log|z-\xi||dzd\xi\lesssim \int_\Box |x|\chi_Adxdy.
\end{equation}
\end{lemma}
\begin{proof}
Notice that
\[
\int_A\int_A |\log|z-\xi||dzd\xi\le\int_\Box\int_\Box
|\log|z-\xi||dzd\xi<\infty\,.
\]
Therefore, by Lemma \ref{petal} applied with $s=|A|$, we can always
assume that $|A|=\epsilon<\epsilon_0$ where $\epsilon_0$ is
sufficiently small. Consider $E_j=A\cap
\{j\epsilon<x<(j+1)\epsilon\}$ and let $\delta_j=|E_j|$,
$I_j=\delta_j/\epsilon, j=-N,\ldots, N, N=[\epsilon^{-1}]$. We have
$\sum_{j=-N}^{N} I_j=1$.

In \eqref{pet}, \begin{equation} \label{eqn1}
 \int_\Box |x|\chi_Adxdy\sim
\epsilon^2\sum_{j=-N}^N |j|I_j+\epsilon^2.
\end{equation}
Indeed, if $j\neq 0,-1$, then
\[
\int_{E_j}|x|\chi_Adxdy\sim \epsilon |j| \delta_j=\epsilon^2 |j|I_j.
\]
For $j=0$ and $j=-1$, we have
\[
\int_{E_j}|x|\chi_Adxdy\le \epsilon \delta_j\le \epsilon^2
\]
and
\[
\int_\Box |x|\chi_Adxdy\lesssim \epsilon^2\sum_{j=-N}^N
|j|I_j+\epsilon^2
\]
follows. To prove a lower bound, notice that
 $\delta_0+\delta_{-1}>\epsilon/2$ implies
\[
\int_{E_0\cup E_{-1}} |x|\chi_Adxdy\gtrsim \epsilon^2
\]
by applying Lemma \ref{petal} with $s=\delta_0+\delta_{-1}$. If
$\delta_0+\delta_{-1}<\epsilon/2$, then $|A\backslash (E_0\cup
E_{-1})|>\epsilon/2$ and
\[
\int_{|x|>\epsilon} |x|\chi_{A}dxdy\gtrsim \epsilon^2.
\]
Therefore, we have
\begin{equation}\label{abel1}
\int_\Box |x|\chi_Adxdy\gtrsim\epsilon^2
\end{equation}
either way. Moreover, from the definition of $\delta_j$ and $I_j$,
we get
\begin{equation}\label{abel2}
\int_\Box |x|\chi_Adxdy\ge\sum_{j\neq \{-1,0\}}\int_\Box
|x|\chi_{E_j}dx\gtrsim \epsilon^2 \sum_{j\neq \{-1,0\}} |j|I_j\,.
\end{equation}
Taking the sum of \eqref{abel1} and \eqref{abel2}, we get
\begin{equation}\label{quest3}
\int_\Box |x|\chi_Adxdy\gtrsim    \sum_{j\neq \{-1,0\}}
|j|I_j+\epsilon^2\gtrsim
       \epsilon^2\sum_{j=-N}^N |j|I_j+\epsilon^2\,
\end{equation}
since $\epsilon^2\ge \epsilon^2I_{0(-1)}$.

\smallskip

 Define the potential $U(z)=\int_A
|\log|z-\xi||d\xi$. For fixed $\{\delta_l\}, l=-N,\ldots,N$, we will
estimate $\max_{E_j} U(z)$ for each $j=-N,\ldots,N$. Let
\[% begin JB edit: I'm not crazy about the parentheticals her; I think they make the argument harder than necessary to follow. I've rewritten in somewhat, but only commented out the bits that I changed so they can be added back in, if necessary,
\max_{E_j}U(z)=U(z_j^*).%\le M_1+M_2+M_3
\]
We can bound the right hand side above by decomposing $U(z_j^*)$ as
\[
 U(z_j^*)\le M_1+M_2+M_3.
\]
%where
The term $M_1$ comes from considering the $\epsilon$-ball around the
point of maximum $z^*_j$. It satisfies
\begin{equation}\label{nnn1}
M_1\lesssim \int_{|z|<\epsilon}|\log|z||dz\sim
\epsilon^2|\log\epsilon|.
\end{equation}
% (this part comes from $\epsilon$-ball around the point of maximum
% $z^*_j$)
The term $M_2$ comes from integrating over $(E_{j-1}\cup E_j \cup
E_{j+1})\backslash B_\epsilon(z_j^*)$. To estimate it, we notice the
following. Consider, e.g, $E_{j-1}$. If
$\delta_{j-1}=|E_{j-1}|<\epsilon^2$, then
\[
\int_{E_{j-1}}|\log|\xi-z_j^*||d\xi\lesssim
\epsilon^2|\log\epsilon|\,,
\]
because the maximizer of the integral in the left-hand side belongs
to a ball $|\xi-z_j^*|\lesssim \epsilon$. On the other hand, if
$|E_{j-1}|>\epsilon^2$, then
\[
\int_{E_{j-1}}|\log|\xi-z_j^*||d\xi\lesssim \epsilon
\int_{0}^{I_{j-1}}|\log y|dy\lesssim \epsilon I_{j-1}(|\log
I_{j-1}|+1)\,,
\]
because the maximizer of the integral belongs to the rectangle of
hight $\sim I_{j-1}$. Arguing similarly for $E_j$ and $E_{j+1}$, we
get
\begin{equation}\label{nnn2}
M_2 \lesssim \epsilon (\epsilon|\log\epsilon|+I_{j}|\log
I_j|+I_{j-1}|\log I_{j-1}|+I_{j+1}|\log
I_{j+1}|+I_{j-1}+I_j+I_{j+1})\lesssim \epsilon\,,
\end{equation}
where the last bound follows from $x|\log x|\lesssim 1$ when $x<1$.
% (this part comes from integrating over $(E_{j-1}\cup E_j \cup
% E_{j+1})\backslash B_\epsilon(z_j^*)$ and taking into account that
% the optimal configurations of, e.g., $E_{j-1}$ will be inside the
% rectangles of heights $\sim I_{j-1}$ provided that
% $E_{j-1}>\epsilon^2$ and inside the ball of radius $\sim \epsilon$
% otherwise.
Finally, the term $M_3$ covers integration over the remaining $E_k$.
It will satisfy
\[
M_3\lesssim \sum_{k:|k-j|>1,|k|\le N }\left( \epsilon^2 |\log
(\epsilon |k-j|)| +\epsilon\int_0^{I_k} |\log (y+\epsilon
|k-j|)|dy\right)\,.
\]
%$\Bigl(\Bigr.$
Indeed,
\[
|\log|z_j^*-\xi||\lesssim  |\log (|k-j|\epsilon+|y_j^*-\xi_2|)|,
\quad \xi\in E_{k}\,,
\]
where $z_j^*=(x_j^*,y_j^*), \xi=(\xi_1,\xi_2)$. Then, we have an
upper bound for the following variational problem
\begin{equation}\label{ihe}
\sup_{\Upsilon_k\subseteq \{k\epsilon<x<(k+1)\epsilon\},
|\Upsilon_k|=|E_k|}\int_{\Upsilon_k} |\log
(|k-j|\epsilon+|\xi_2|)|d\xi\lesssim \epsilon^2 |\log (\epsilon
|k-j|)| +\epsilon\int_0^{I_k} |\log (y+\epsilon |k-j|)|dy\,,
\end{equation}
where the first term comes from the case when $|E_k|\le \epsilon^2$,
the second one comes from the other case and the observation that
the optimal configuration $\Upsilon_k^*$ is a rectangle
 of the size $\sim I_k$.

  Integration gives
\[
\int \log ydy=y\log y-y+C
\]
and so
\[
\int_0^{I_k} |\log (y+\epsilon |k-j|)|dy\le\int_0^{I_k} |\log y|dy
\lesssim  I_k|\log I_k|+I_k.
\]
The first term in the right-hand side of \eqref{ihe} gives
\[
\epsilon^2 \sum_{k:|k-j|>1, |k|\le N}|\log (\epsilon|k-j|)|\lesssim
\epsilon^2\sum_{1<|l|\le 2N}|\log (\epsilon |l|)|\lesssim
\epsilon\int_{|x|<1}|\log |x||dx\lesssim \epsilon
\]
by making comparison to an integral.
 Finally, summing over $k$, we have
\[
M_3\lesssim \epsilon +\epsilon\sum_{k=-N}^N I_k|\log I_k|.
\]
Taking into account \eqref{nnn1},\eqref{nnn2}, we get
\[
\max_{z\in A} U(z)\lesssim \epsilon+\epsilon\sum_{k=-N}^N I_k|\log
I_k|.
\]
 Then,
\[
\int_A\int_A |\log|z-\xi||dzd\xi \lesssim
\epsilon^2+\epsilon^2\sum_{k=-N}^N I_k |\log I_k|.
\]
The trivial estimate
\[
|u\log u|\le C(\gamma)u^\gamma, \quad 0<\gamma<1, 0<u<1
\]
implies
\[
\sum_{j\neq 0} I_j |\log I_j|\lesssim \sum_{j\neq 0}
I_j^\gamma=\sum_{j\neq 0} (jI_j)^\gamma j^{-\gamma}\le
\left(\sum_{j\neq 0}|j|I_j\right)^{1/p}\left(\sum_{j\neq
0}|j|^{-\gamma p'}\right)^{1/p'} \lesssim \left(\sum_{j\neq
0}|j|I_j\right)^{\gamma}
\]
by H\"older inequality with $p=1/\gamma$ and $\gamma>0.5$. We have
\[
\epsilon^2 I_0|\log I_0|+\epsilon^2 \left(\sum_{j\neq 0}
|j|I_j\right)^{\gamma}\lesssim \epsilon^2+\epsilon^2 \sum_{j\neq 0}
|j|I_j
\]
and application of \eqref{eqn1} finishes the proof of Lemma
\ref{lg}.
\end{proof}

\end{appendix}
{\Large \part*{Acknowledgement}} The research of JB was supported by
the RTG grant NSF-DMS-1147523. The work of SD done in the second
part of the paper was supported by RSF-14-21-00025 and his research
on the rest of the paper was supported by the grants
NSF-DMS-1464479, NSF-DMS-1067413.

\end{document}